\documentclass[12pt]{amsart}

\usepackage{color}
\usepackage{amsmath}
\usepackage{amsthm}
\usepackage{amssymb}
\usepackage{graphicx}
\usepackage{enumerate}
\usepackage{amsfonts}
\usepackage{mathrsfs}
\usepackage{parskip}
\usepackage{mathdots}
\usepackage{color}

\numberwithin{equation}{section}
\theoremstyle{plain}
\newtheorem{Proposition}[equation]{Proposition}
\newtheorem{Corollary}[equation]{Corollary}
\newtheorem*{Corollary*}{Corollary}
\newtheorem{Theorem}[equation]{Theorem}
\newtheorem*{Theorem*}{Theorem}
\newtheorem{Lemma}[equation]{Lemma}
\theoremstyle{definition}
\newtheorem{Definition}[equation]{Definition}

\newtheorem{Example}[equation]{Example}
\newtheorem{Remark}[equation]{Remark}

\def\HH{\mathscr{H}}
\def\MM{\mathscr{M}}
\def\DD{\mathscr{D}}
\def\C{\mathbb{C}}

\def\D{\mathbb{D}}
\def\T{\mathbb{T}}
\def\N{\mathbb{N}}

\newcommand{\Span}{\operatorname{span}}

\renewcommand{\ker}{\operatorname{Ker}}

\newcommand{\Rng}{\operatorname{Rng}}
\newcommand{\Pol}{\operatorname{Pol}}
\newcommand{\beqa}{\begin{eqnarray*}}
\newcommand{\eeqa}{\end{eqnarray*}}
\newcommand{\dst}{\displaystyle}

\title[Range spaces of Toeplitz operators]{Range spaces of co-analytic Toeplitz operators}

\author[Fricain]{Emmanuel Fricain}
 \address{Laboratoire Paul Painlev\'e, Universit\'e Lille 1, 59 655 Villeneuve d'Ascq C\'edex }
 \email{emmanuel.fricain@math.univ-lille1.fr}

\author[Hartmann]{Andreas Hartmann}
\address{Institut de Math\'ematiques de Bordeaux, Universit\'e Bordeaux 1, 351 cours de la Lib\'eration 33405 Talence C\'edex, France}
\email{Andreas.Hartmann@math.u-bordeaux1.fr}

\author[Ross]{William T. Ross}
	\address{Department of Mathematics and Computer Science, University of Richmond, Richmond, VA 23173, USA}
	\email{wross@richmond.edu}

\keywords{Toeplitz operators, boundary behavior,
de Branges-Rovnyak spaces, non-extreme points, kernel functions, corona pairs}

\subjclass[2010]{30J05, 30H10, 46E22}

\begin{document}

\begin{abstract}
We discuss the range spaces of  Toeplitz operators with co-analytic symbols where we focus on the boundary behavior of the functions in these spaces as well as a natural orthogonal decomposition of this range. 
\end{abstract}

\maketitle

\section{Introduction}

In this paper we examine the range of co-analytic Toeplitz operators on the classical Hardy space $H^2$ of the  open unit disk $\D$.  In particular, we explore both the boundary behavior of functions in the range as well as a natural orthogonal decomposition of the range in a suitable Hilbert space structure. 

To explain our results, let $T_{\varphi}$ be the Toeplitz operator on $H^2$ with symbol $\varphi \in L^{\infty}$ and define its range space 
$$\MM(\varphi) :=T_{\varphi}H^2.$$
This range space is endowed with the inner product $\langle \cdot, \cdot \rangle_{\varphi}$ defined by 
$$
 \langle T_{\varphi}f,T_{\varphi}g\rangle_{\varphi} :=\langle f,g\rangle_{H^2}, \quad f,g \in H^2 \ominus \ker T_{\varphi},$$ 
 where $\langle \cdot, \cdot \rangle_{H^2}$ is the inner product in $H^2$. We remind the reader of some standard facts in the next section.
  

When $a \in H^{\infty}$, the bounded analytic functions on $\D$, and is outer, the co-analytic Toeplitz operator $T_{\overline{a}}$ is injective with dense range $\MM(\overline{a})$ in $H^2$ (Proposition \ref{Toe-facts}). In this case, the corresponding inner product $\langle \cdot, \cdot \rangle_{\overline{a}}$ on $\MM(\overline{a})$ becomes 
\begin{equation}\label{bbdggdyyd}
\langle T_{\overline{a}}f,T_{\overline{a}}g\rangle_{\overline{a}}=\langle f,g\rangle_{H^2}, \quad f,g\in H^2.
\end{equation}

Many properties of Toeplitz operators have been well investigated (see e.g.\ \cite{BS, Nik1, Nik2}). 
The less studied range spaces make important connections with the de Branges--Rovnyak spaces \cite{FM, Sa}, and the paper \cite{MR1065054} characterizes the common range of the co-analytic Toeplitz operators.  In this paper we begin a more focussed discussion of $\MM(\overline{a})$ and its various properties. 


Our first goal is to study the boundary behavior of functions in $\MM(\overline{a})$. Functions, along with their derivatives, in the so-called sub-Hardy Hilbert spaces can have more regularity at particular $\zeta_0$ on the unit circle $\T$ than generic functions in $H^2$. 
 Broadly speaking, these type of results say that if certain conditions are satisfied, then {\em every} function in a given sub-Hardy Hilbert space has a non-tangential limit at a particular $\zeta_0 \in \T$. 

As a specific example of these kind of results, suppose that $I$ is an inner function factored (canonically) as 
$I = B s_{\mu}$, 
where the first factor $B$ is a Blaschke product with zeros $\{a_n\}_{n \geqslant 1} \subset \D$ while the second factor $s_{\mu}$ is a singular inner function with corresponding positive measure $\mu$ on $\T$ with $\mu \perp d \theta$ \cite{Duren, Garnett}. One can define the well-studied model space 
\begin{equation}\label{MS}
K_I:= H^2\ominus IH^2= (I H^2)^{\perp}
\end{equation}
 \cite{Nik, Nik1, Nik2}. A theorem of Ahern and Clark \cite{AC70} says that if $\zeta_0 \in \T$ and $N \in \N_0 := \N \cup \{0\}$, then every $f \in K_I$, along with the derivatives $f', \ldots, f^{(N)}$,  has a finite non-tangential limit at $\zeta_0$ if and only if 
\begin{equation}\label{AC-cont}
\sum_{n \geqslant 1} \frac{1 - |a_n|}{|\zeta_0 - a_n|^{2 N + 2}} + \int_{\T} \frac{d \mu(\xi)}{|\zeta_0 - \xi|^{2 N + 2}}  < \infty.
\end{equation}
This work was extended by Fricain and Mashreghi \cite{MR2390675,FM08} to the closely related de Branges-Rovnyak spaces $\HH(b)$ (defined below), where $b$ is in the closed unit ball
$H^{\infty}_1$  of $H^{\infty}$, and factored (canonically) as $b = B s_{\mu} b_0$, where $B s_{\mu}$ is the inner factor of $b$ and $b_0$ its outer factor. 
Here the necessary and sufficient condition that every $f \in \HH(b)$, along with $f', \ldots, f^{(N)}$,  has a finite non-tangential limit at $\zeta_0$ becomes 
\begin{equation}\label{FM-cont}
\sum_{n \geqslant 1} \frac{1 - |a_n|}{|\zeta_0 - a_n|^{2 N + 2}} + \int_{\T} \frac{d \mu(\xi)}{|\zeta_0 - \xi|^{2 N + 2}}  +  \int_{0}^{2 \pi} \frac{|\log|b(e^{i \theta})||}{|\zeta_0 - e^{i \theta}|^{2 N + 2}} d\theta
\end{equation}
is finite. See \cite{Sa, Bolotnikov-Kheifets} for some related results.

The technique originally used by Ahern and Clark, and extended by others, to discover conditions like \eqref{AC-cont} was to control the norm of the reproducing kernels as one approached the boundary point $\zeta_0 \in \T$. We will explore this Ahern-Clark technique in a broader setting to not only capture the boundary behavior of functions in the range spaces $\MM(\overline{a})$, the primary focus of this paper, but also the de Branges-Rovnyak spaces $\HH(b)$, and even the harmonically weighted Dirichlet spaces $\mathscr{D}(\mu)$.

To describe the boundary behavior in $\MM(\overline{a})$, we first observe that we can always assume that $a$ is an outer function (Proposition~\ref{Prop:outer-wlog}). Furthermore, in Theorem \ref{thmACbis} and Corollary \ref{AC}, we will show that if $\zeta_0 \in \T$ and $N\in \N_0$, then every $f \in \MM(\overline{a})$, along with $f', f'', \ldots, f^{(N)}$, has a finite non-tangential limit at $\zeta_0$ if and only if 
\begin{equation}\label{77777}
\int_{0}^{2 \pi} \frac{|a(e^{i \theta})|^{2}}{|e^{i \theta} - \zeta_0|^{2 N + 2}} d \theta < \infty.
\end{equation}
Obviously,
the convergence of the integral in \eqref{77777} depends on the strength of the zero of $a$ at $\zeta_0$. We will use this observation to show (Proposition \ref{99999}) that there is no point $\zeta_0 \in \T$ for which every function in $\MM(\overline{a})$ has an analytic continuation to an open neighborhood of $\zeta_0$. This is in contrast to the model spaces $K_I$ where, under certain circumstances,  every function in $K_I$ has an analytic continuation across a portion of $\T$ \cite{DSS}. 
We point out that our boundary behavior results for $\MM(\overline{a})$ make connections to similar types of results for  $T_{\overline{a}} K_I$ \cite{HR12}. 

To discuss the internal Hilbert space structure of $\MM(\overline{a})$, we first observe (Proposition \ref{Prop:containment}) that $\MM(a) \subset \MM(\overline{a})$ with contractive inclusion. The space $\MM(a)$ has an obvious description as 
$$a H^2 = \{a f: f \in H^2\},$$ 
and we are interested in how $\MM(a)$ completes to $\MM(\overline{a})$ when $\MM(a)$ is
complemented in $\MM(\overline{a})$. This happens when $\MM(a)$ is closed in the topology of $\MM(\overline{a})$, which takes place when the Toeplitz operator $T_{\overline{a}/a}$ is surjective (Proposition \ref{04983}) \cite{HSS}. In this case we have an orthogonal decomposition 
$$\MM(\overline{a}) = \MM(a) \oplus_{\overline{a}} K$$
for some closed subspace $K$ of $\MM(\overline{a})$. Here $\oplus_{\overline{a}}$ denotes the orthogonal sum in the inner product $\langle \cdot, \cdot \rangle_{\overline{a}}$. 
To identify the summand $K$, we will show that 
$$\MM(\overline{a}) = \MM(a) \oplus_{\overline{a}} T_{\overline{a}} \ker T_{\overline{a}/a}$$ and then proceed to use the well developed theory of the kernels of Toeplitz operators from \cite{HSS, HS, Ha86, Ha90, Hi, MR1300218} to identify, in certain cases, $T_{\overline{a}} \ker T_{\overline{a}/a}$. Our previous results on the boundary behavior naturally come into play here.
Indeed, when \eqref{77777} is satisfied, then point evaluation kernels as well as their derivatives
up to order $N$ are elements of $K$ (see Proposition 
\ref{Prop:kernel-boundary-orthogonalcomplement}) and, in certain situations, span the
complementary space $K$ (Corollary \ref{kernelspan}).

In particular, but not all, cases, the decomposition takes the form 
$\MM(\overline{a}) = \MM(a) \oplus_{\overline{a}} K_I$, where $K_I$ is a model space corresponding to an inner function $I$ associated with $a$.

Finally, we will use our techniques to generalize  the results from \cite{FHR,LanNow} to decompose the de Branges Rovnyak spaces $\HH(b)$ for certain $b$ (Theorem \ref{Thm:Hb-decomposition-orthogonale}).

\section{Some reminders}


Let $H^2$ denote the classical Hardy space of the unit disk $\D$ \cite{Duren, Garnett} endowed with the standard $L^2$ inner product 
$$\langle f, g\rangle_{H^2} := \int_{\T} f \overline{g} dm,$$ 
where $m$ is normalized Lebesgue measure on $\T$.

Recall that $H^2$ is a reproducing kernel Hilbert space 
with reproducing (Cauchy) kernel 
\begin{equation}\label{Cauchy-kernel}
k_{\lambda}(z) := \frac{1}{1 - \overline{\lambda} z}, \quad \lambda, z \in \D,
\end{equation} 
meaning that 
$$f(\lambda) = \langle f, k_{\lambda}\rangle_{H^2}, \quad f \in H^2, \lambda \in \D.$$
Let
$P_{+}: L^2 \to H^2$ the usual (orthogonal) {\em Riesz projection} given by the formula
$$(P_{+} f)(\lambda) = \langle f, k_{\lambda}\rangle_{L^2}, \quad f \in L^2, \lambda \in \D.$$

 If $n \in \N_0$, $\lambda \in \D$, and 
\[
k_{\lambda,n}(z) :=  \frac{n! z^n}{(1-\overline{\lambda} z)^{n+1}},
\]
then $k_{\lambda, n}$ is the reproducing kernel for the $n$-th derivative at $\lambda$ in that 
\begin{equation}\label{230487534857}
f^{(n)}(\lambda)=\langle f,k_{\lambda,n}\rangle_{H^2}, \quad f \in H^2.
\end{equation}


For a symbol $\varphi \in L^{\infty}$, the space of essentially bounded Lebesgue measurable functions on $\T$, define the {\em Toeplitz operator} $T_{\varphi}$ on $H^2$ by 
$$T_{\varphi} f := P_+(\varphi f), \quad f \in H^2.$$ When $\varphi \in H^{\infty}$, $T_{\varphi}$ is called an {\em analytic Toeplitz operator} (sometimes called a Laurent operator), and is given by the simple formula $T_{\varphi} f = \varphi f$, while $T^{*}_{\varphi} = T_{\overline{\varphi}}$ is called a {\em co-analytic Toeplitz operator}.

We gather up the following useful facts about Toeplitz operators. See \cite{FM, Nik, Nik1} for more details. 

\begin{Proposition}\label{Toe-facts}
Let $\varphi, \psi \in L^{\infty}$. 
\begin{enumerate}
\item If $\varphi \in H^{\infty}$, then $T_{\overline{\varphi}} k_{\lambda} = \overline{\varphi(\lambda)} k_{\lambda}$ for every $\lambda \in \D$. 
\item If $\varphi \in H^{\infty}$ and outer, then the Toeplitz operators $T_{\varphi}, T_{\overline{\varphi}}$, and $T_{\varphi/\overline{\varphi}}$ are injective. 
\item If at least one of $\varphi, \psi$ belongs to $H^{\infty}$, then $T_{\overline{\psi}} T_{\varphi} = T_{\overline{\psi} \varphi}$.
\item If $\varphi \in H^{\infty}$ and $I$ is the inner factor of $\varphi$, then 
$$\ker{T_{\overline{\varphi}}} = K_{I}.$$ 
\item If $\varphi \in H^{\infty}$ and $I$ is inner, then $T_{\overline{\varphi}} K_I \subset K_I$. 
\end{enumerate}
\end{Proposition}

The kernel $\ker{T_{\varphi}}$ of a Toeplitz operator has been well studied and will play an important role in our orthogonal decomposition. Let us recall a few results in this area. A closed linear subspace $M$ of $H^2$ is said to be {\em nearly invariant} if
$$f \in M, \; \; f(0) = 0 \implies \frac{f}{z} \in M.$$ 
We will only consider the non-trivial nearly invariant subspaces of $H^2$: $\{0\} \subsetneq M \subsetneq H^2$.

\begin{Theorem}[Hitt \cite{Hi}, Sarason \cite{MR1300218}]\label{Hitt-Thm}
Let $M$ be a non-trivial nearly invariant subspace of $H^2$.
If $\gamma$ is the unique solution to the extremal problem
$$\sup\{\Re g(0):g\in M, \|g\|_{H^2}\leqslant 1\},$$ then there is
an inner function $I$ with $I(0)=0$ such that
$$
 M=\gamma K_I.
$$
Furthermore,  $\gamma $ is an isometric multiplier from $K_I$ onto $\gamma K_I$
and can be written as
\[
 \gamma=\frac{\alpha}{1- \beta_0 I},
\]
where $\alpha, \beta_0 \in H^{\infty}$ and $|\alpha|^2 + |\beta_0|^2 = 1$ a.e.\! on $\T$.

Conversely, every space of the form $M=\gamma K_I$, with 
$$\gamma=\frac{\dst \alpha}
{\dst 1-I\beta_0},$$ $\alpha, \beta_0 \in H^{\infty}$, $|\alpha|^2+|\beta_0|^2=1$ a.e.\ on $\T$, and $I$ inner with $I(0)=0$,
is nearly invariant with associated extremal function $\gamma$.
\end{Theorem}

The parameters $\gamma$ and $\beta=I\beta_0$ are related by the following formula from \cite{MR1300218}:
\begin{equation}\label{beta}
\frac{1 + \beta(z)}{1 - \beta(z)} = \int_{\T} \frac{\zeta + z}{\zeta - z} |\gamma(\zeta)|^2 dm(\zeta), \quad z \in \D.
\end{equation}

Clearly, when $\varphi\in L^{\infty}$, then $\ker T_{\varphi}$ is nearly invariant. Hayashi identified
those nearly invariant subspaces which are kernels of Toeplitz operators. With the notation from Theorem \ref{Hitt-Thm}, set  
$$\gamma_0 :=\frac{\dst\alpha}{\dst 1-\beta_0}.$$

\begin{Theorem}[Hayashi \cite{Ha90}]
A non trivial nearly invariant subspace $M$ is the kernel of a Toeplitz operator if and only if
$\gamma_0^2$ is rigid in $H^1$.
\end{Theorem}

The $H^1$ function $\gamma_{0}^{2}$ is said to be {\em rigid} if the only $H^1$ functions having the same argument as $\gamma_{0}^{2}$ almost everywhere on $\T$ are $\{c \gamma_{0}^{2}: c > 0\}$. One can show that 
if $g$ and  $1/g$ both belong to $H^1$ then $g$ is rigid. 
The converse is not always true. 

Observe that the extremal function for the kernel of a Toeplitz operator is necessarily outer (one can always divide out the inner factor). In particular, for this situation, $\alpha$ is always outer. 

If $\gamma$ is the extremal function for $\ker{T_{\varphi}}$, with associated inner function $I$, then
\begin{equation}\label{kerker}
 \ker T_{\varphi}= \gamma K_I=\ker{T_{\overline{I \gamma}/\gamma}} .
\end{equation}

Note that when $\gamma_0^2$ is rigid, then $T_{\overline{\gamma_0}/\gamma_0}$ is 
injective  \cite[Theorem X-2]{Sa}. In this paper we will also need the  stronger property, namely the invertibility
of $T_{\overline{\gamma_0}/\gamma_0}$. This is characterized in \cite{HSS} by the well-known $(A_2)$-condition.

\begin{Theorem}\label{7796316}
With the notation above, suppose that $\ker{T_{\varphi}} \not = \{0\}$. Then the Toeplitz operator $T_{\varphi}$ is surjective if and only if 
$|\gamma_{0}|^2$ is an $(A_2)$ weight, meaning 
\begin{equation}\label{A2}
\sup_{J}\left( \frac{1}{J} \int_{J} |\gamma_0|^2 dm\right) \left(\frac{1}{J} \int_{J} |\gamma_{0}|^{-2} dm\right) < \infty,
\end{equation}
where the supremum above is taken over all arcs $J \subset \T$.
\end{Theorem}


\section{Range spaces}\label{Sec:the-range-space}

For a bounded linear operator $A: H^2 \to H^2$, define the {\em range space} 
$$\MM(A) := A H^2$$  and endow it with the {\em range norm}
\begin{equation}\label{Rng-norm}
\|A f\|_{\MM(A)} :=\|f\|_{H^2}, \quad f \in H^2 \ominus \ker{A}.
\end{equation} The induced inner product 
$$\langle A f, A g\rangle_{\MM(A)} := \langle f, g \rangle_{\MM(A)}, \quad f, g \in H^2 \ominus \ker{A}$$
makes $\MM(A)$ a Hilbert space and 
makes $A$ a partial isometry with initial space $H^2 \ominus \ker{A}$ and final space $A H^2$. 
In fact, using the identity $(\ker{A})^{\perp} = (\mbox{Rng} \,A^{*})^{-}$, we see that
\begin{eqnarray}\label{AAstar}
 \langle f,AA^*g\rangle_{\MM(A)}=\langle f,g\rangle_{H^2}, \quad f \in \MM(A), \, g \in H^2.
\end{eqnarray}
These range spaces $\MM(A)$, as well as their complementary spaces, were formally introduced by Sarason \cite{Sa} though they appeared earlier in the context of square summable power series in the work of de Branges and Rovnyak \cite{MR0244795, MR0215065}. We will discuss this connection in a moment.

Since $\MM(A)$ is boundedly contained in $H^2$, meaning that the inclusion operator is bounded, we see that for fixed $n  \in \N_0$ and $\lambda \in \D$,  the linear functional $f\mapsto f^{(n)}(\lambda)$ is continuous on $\MM(A)$. By the Riesz representation theorem, this functional is given by a reproducing kernel $k_{\lambda,n}^{\MM(A)}\in\MM(A)$, that is to say,
\[
f^{(n)}(\lambda)=\langle f,k_{\lambda,n}^{\MM(A)}\rangle_{\MM(A)},\qquad f\in\MM(A).
\]
\begin{Proposition}\label{RK-formula1}
For fixed $\lambda \in \D$ and $n\in \N_0$, we have 
$$k_{\lambda,n}^{\MM(A)}=AA^*k_{\lambda,n}.
$$
\end{Proposition}

\begin{proof}
For any $f \in \MM(A)$, use \eqref{AAstar} along with \eqref{230487534857} to get
$$
\langle f,AA^*k_{\lambda,n}\rangle_{\MM(A)} = \langle f,k_{\lambda,n}\rangle_{H^2} =f^{(n)}(\lambda). \qedhere
$$
\end{proof}

When $A$ is a co-analytic Toeplitz operator $T_{\overline{a}}$ ($a\in H^\infty$), we obtain a special form for the reproducing kernel. 

\begin{Corollary}\label{RK-formula}
For each $\lambda \in \D$ and $n \in \N_0$ we have 
$$k_{\lambda,n}^{\MM(T_{\overline{a}})}=T_{\overline{a}}a
k_{\lambda,n}=T_{|a|^2}k_{\lambda,n}.$$
\end{Corollary}

\begin{proof}
Observe that $T_{\overline{a}}^*=T_a$ and apply Proposition \ref{RK-formula1} and Proposition \ref{Toe-facts}(2).
\end{proof}

\begin{Remark}\label{rwerwer}
Since the range space of a co-analytic Toeplitz operator is the primary focus on this paper, we will use the less cumbersome notation 
$$\MM(a) := \MM(T_{a}), \qquad \MM(\overline{a}) := \MM(T_{\overline{a}}),$$ 
$$\langle \cdot, \cdot \rangle_{\overline{a}} := \langle \cdot, \cdot \rangle_{\MM(T_{\overline{a}})},$$
$$k_{\lambda, n}^{\overline{a}} := k_{\lambda,n}^{\MM(T_{\overline{a}})}, \qquad k_{\lambda}^{\overline{a}} := k_{\lambda, 0}^{\overline{a}}.$$

\end{Remark}

Let us mention a few more structural details concerning $\MM(\overline{a})$. For any $a \in H^{\infty}$, let $a_0$ be the outer factor of $a$. 


\begin{Proposition}{\cite[Corollary 16.8]{FM}}\label{Prop:outer-wlog}
$\MM(\overline a)=\MM(\overline{a_0})$ as Hilbert spaces.
\end{Proposition}


\begin{Remark}\label{R-a-outer}
Thus, when discussing $\MM(\overline{a})$ spaces, we can always assume that $a=a_0$ is outer. 
\end{Remark}



\begin{Proposition}{\cite{FM, Sa}}\label{Prop:containment}
For $a \in H^{\infty}$  we have $\MM(a) \subset \MM(\overline{a})$ and the inclusion is contractive.
\end{Proposition}


The previous proposition can be seen from from the simple identity  $T_{a} = T_{\overline{a}} T_{a/\overline{a}}$ which
we will use later.

To connect the results of this paper with those of 
\cite{FHR, LanNow}, let us briefly recall some facts about the de Branges-Rovnyak spaces \cite{FM, Sa}.
For $b \in H^{\infty}_{1} = \{f \in H^{\infty}: \|f\|_{\infty} \leqslant 1\}$, the closed unit ball
in $H^{\infty}$, define 
$$A:=(I-T_bT_{\overline{b}})^{1/2}.$$ 
The {\em de Branges-Rovnyak space} $\HH(b)$ is defined to be
\begin{equation}\label{DRSp}
\HH(b) := \MM(A),
\end{equation}
endowed with the range norm from \eqref{Rng-norm}.

\begin{Remark}
In a similar vein to Remark \ref{rwerwer}, we set
$$\langle \cdot, \cdot \rangle_{b} := \langle \cdot, \cdot \rangle_{\MM(A)}, \qquad k_{\lambda,n}^{b} := k_{\lambda,n}^{\MM(A)},\qquad k_\lambda^b:=k_{\lambda,0}^b,$$
when $A=(I-T_bT_{\overline{b}})^{1/2}$ and $n\in\N_0$.
\end{Remark}

When $\|b\|_{\infty} < 1$ it turns out that  $\HH(b) = H^2$ with an equivalent norm. When $b = I$ is an inner function, then 
$\HH(I) = K_I$ is one of the model spaces from \eqref{MS} endowed with the $H^2$ norm. 

Suppose $a \in H^{\infty}_{1}$ is outer and satisfies $\log (1 - |a|) \in L^1=L^1(\T, m)$. 
This log integrability condition is equivalent to the fact that $a$ is a {\em non-extreme} point of $H^{\infty}_{1}$. Let $b$ be the outer function, unique if we require the additional condition that $b(0) > 0$, which satisfies 
$$|a|^2 + |b|^2 = 1 \quad \mbox{a.e.\! on $\T$}.$$ We call $b$, necessarily in $H^{\infty}_{1}$, the {\em Pythagorean mate} for $a$. If $\HH(b)$ is the associated de Branges-Rovnyak space from \eqref{DRSp}, 
it is known \cite[p.~24]{Sa} that 
$$\MM(a) \subset \MM(\overline{a}) \subset \HH(b),$$ though neither $\MM(a)$ nor $\MM(\overline{a})$ is necessarily closed in $\HH(b)$. Still, $\MM(\overline{a})$ is always dense in $\HH(b)$. Furthermore,  when $(a,b)$ is a {\em corona pair}, that is to say, 
\begin{equation}\label{CP}
\inf\{|a(z)| + |b(z)|: z \in \D\} > 0,
\end{equation}
 then $\HH(b)=\MM(\overline{a})$ \cite[Theorem 28.7]{FM} or \cite{Sa}. The equality $\MM(\overline{a}) = \HH(b)$ is a set equality but the norms, though equivalent by the closed graph theorem, need not be equal.



\section{Boundary behavior in sub-Hardy Hilbert spaces}\label{Section2}


While the focus of this paper is the boundary behavior of functions in $\MM(\overline{a})$, or more generally the range spaces $\MM(A)$, our discussion of boundary behavior can be broadened to a class of ``admissible'' reproducing kernel Hilbert spaces of analytic functions on $\D$. 

To get started, let $\HH$ be a Hilbert space of analytic functions on $\D$ with norm $\|\cdot\|_{\HH}$ such that for each $\lambda \in \D$, the evaluation functional $f \mapsto f(\lambda)$ is continuous on $\HH$. By the Riesz representation theorem, there is a $k_{\lambda}^{\HH} \in \HH$ such that 
$$f(\lambda) = \langle f, k_{\lambda}^{\HH}\rangle_{\HH}.$$ This function $k_{\lambda}^{\HH}(z)$, called the {\em reproducing kernel} for $\HH$, is an analytic function of $z$ and a co-analytic function of $\lambda$. The space $\HH$ with such a kernel function is called a {\em reproducing kernel Hilbert space} \cite{Paulsen}. 

For each $n \in \N_0$ it follows that the linear functional $f \mapsto f^{(n)}(\lambda)$ is also continuous on $\HH$ and thus given by a reproducing kernel $k_{\lambda,j}^{\HH} \in \HH$: $$f^{(j)}(\lambda) = \langle f, k_{\lambda,j}^{\HH}\rangle_{\HH}, \quad f \in \HH, \lambda \in \D.$$ A brief argument from \cite[p. 911]{FM} will show that 
\begin{equation}\label{eq:derivative-kernel}
k_{\lambda,j}^{\HH}=\frac{\partial^j}{\partial\overline\lambda^j}k_\lambda^{\HH}.
\end{equation}
When $j=0$ we set $k_{\lambda}^{\HH}:=k_{\lambda,0}^{\HH}$.

Define the following linear transformations $T$ and $B$ on $\mathscr{O}(\D)$ (the  vector space of analytic functions on $\D)$ by 
$$(T f)(z) = z f(z), \quad (B f)(z) = \frac{f(z) - f(0)}{z}.$$
Observe that $S := T|_{H^2}$ is the well-known unilateral shift operator on $H^2$ and $S^{*} = B|_{H^2}$ is the equally well-known backward shift. Observe further that, in terms of Toeplitz operators on $H^2$, we have $S=T_z$ and $S^*=T_{\overline z}$.

\begin{Definition}\label{Def-admissible}
A reproducing kernel Hilbert space 
$\mathscr{H}$ of analytic functions on $\D$ satisfying the two conditions
\begin{enumerate}
\item $B\HH\subset \HH$ and $\|B\|_{\HH \to \HH}\leqslant 1$,
\item $\sigma_p(X_{\HH}^*)\subset \D$, where $X_{\HH} := B|_{\HH}$,
\end{enumerate}
will be called {\em admissible}. In the above, $\sigma_{p}(X_{\HH}^{*})$ is the point spectrum 
of the operator $X_{\HH}^{*}$. 
\end{Definition}

We will discuss some examples, such as $\MM(\overline{a})$, $\HH(b)$, and $\DD(\mu)$ towards the end of this section.

The following result, valid beyond the setting of admissible spaces (see \cite[p. 912]{FM} for an alternate proof given in terms of $\HH(b)$ spaces),
gives us a useful formula for the reproducing kernels $k_{\lambda, j}^{\HH}$.
\begin{Lemma}\label{Lem:formule-noyau-reproduisant-derive}
Let $\HH$ be a reproducing kernel Hilbert space of analytic functions on $\D$ such that
$B\HH\subset \HH$ and $\|B\|\leqslant 1$.
Then for each $j \in \N_0$ and $\lambda \in \D$ we have 
\begin{eqnarray}\label{repklambdaH}
k_{\lambda,j}^{\HH}=j!(I-\overline\lambda X_{\HH}^*)^{-(j+1)}{X_{\HH}^*}^j k_0^{\HH}.
\end{eqnarray}
\end{Lemma}

\begin{proof}
We first establish \eqref{repklambdaH} when $j=0$. 
Since $B$ is a contraction, the operator $(I-\overline{\lambda}X_{\HH}^*)$ is invertible when $\lambda\in \D$ and
the formula in \eqref{repklambdaH}, for $j = 0$, is equivalent to the identity 
$
(I-\overline{\lambda}X_{\HH}^*)k_{\lambda}^{\HH}
=k_{0}^{\HH}.
$ Observe how this identity holds if and only if for every $f\in \HH$, 
$$\langle f, (I-\overline{\lambda}X_{\HH}^*)k_{\lambda}^{\HH}\rangle_{\HH}
=\langle f, k_{0}^{\HH}\rangle_{\HH}=f(0).$$

To prove this last identity, observe that 
\begin{align*}
 \langle f, (I-\overline{\lambda}X_{\HH}^*)k_{\lambda}^{\HH}\rangle_{\HH}
 & = \langle f, k_{\lambda}^{\HH}\rangle _{\HH} - \lambda \langle f, X_{\HH}^{*} k_{\lambda}^{\HH}\rangle_{\HH}\\
 & = f(\lambda) - \lambda \langle X_{\HH} f, k_{\lambda}^{\HH}\rangle_{\HH}\\
 & =f(\lambda)-\lambda\frac{f(\lambda)-f(0)}{\lambda}\\
 & =f(0).
\end{align*}
This proves \eqref{repklambdaH} when $j=0$.

The formula for $k_{\lambda,j}^{\HH}$ now follows from \eqref{eq:derivative-kernel} by differentiating the identity 
$$k_{\lambda}^{\HH}=(I-\overline\lambda X_{\HH}^*)^{-1} k_0^{\HH}$$
$j$ times with respect to the variable $\overline\lambda$. 
\end{proof}


We are now ready to state the main result of this section. For fixed $\zeta_{0} \in \T$ and $\alpha > 1$ let 
$$\Gamma_{\alpha}(\zeta_0) := \big\{z \in \D: |z - \zeta_0| < \alpha (1 -|z|)\big \}$$ be a standard Stolz domain anchored at $\zeta_0$. 
We say that an $f \in \mathscr{O}(\D)$ has a finite {\em non-tangential limit} $L$ at $\zeta_0$ if $f(z) \to L$ whenever $z \to \zeta_0$ within any Stolz domain $\Gamma_{\alpha}(\zeta_0)$. 
When $\alpha=1$, $\Gamma_{1}(\zeta_0)$
degenerates to the radius connecting $0$ and $\zeta_0$ and the limit within $\Gamma_{1}(\zeta_0)$ becomes a radial limit. The non-tangential limit $L$ is denoted by $L = f(\zeta_0)$. 

The following result was inspired by an operator theory result of Ahern and Clark \cite{AC70} where they discussed non-tangential limits of functions in the classical model spaces $K_{I}$.

\begin{Theorem}\label{thmACbis}
Let $\HH$ be an admissible space, $\zeta_0\in\T$, and $N \in \N_0$. Then the following are equivalent:
\begin{enumerate}[(i)]
\item For every $f\in \HH$, the functions $f, f', f'', \ldots, f^{(N)}$ have finite non-tangential limits at $\zeta_0$. 
\item 
For each fixed $\alpha > 1$, we have $$\sup\{\|k_{\lambda,N}^{\HH}\|_{\HH}: \lambda \in \Gamma_{\alpha}(\zeta_0)\}  < \infty.$$
\item  ${X_{\HH}^*}^N k_0^{\HH} \in \Rng(I-\overline\zeta_0 X_{\HH}^*)^{N+1}$. 
\end{enumerate}
Moreover, if any one of the above equivalent conditions hold then
\begin{equation}\label{eq:noyau-reproduisant-derivee-n-ieme-range-of-operatorXbstar}
(I-\overline{\zeta_0}X_{\HH}^*)^{N+1}k_{\zeta_0,N}^{\HH}=N! {X_{\HH}^*}^N k_0^{\HH},
\end{equation}
where $k_{\zeta_0,N}^{\HH}\in \HH$ and satisfies
\[
f^{(N)}(\zeta_0)=\langle f,k_{\zeta_0,N}^{\HH}\rangle_{\HH},\qquad  f\in \HH.
\]
\end{Theorem}

The proof of this requires the following technical lemma from \cite[Cor. 21.22]{FM} (see also \cite{MR2390675}) which generalizes an operator theory result of Ahern and Clark \cite{AC70}.

\begin{Lemma}\label{h2723n2340}
Let $T$ be a contraction on a Hilbert space $\HH$, $\zeta \in \T$, and $\{\lambda_{n}\}_{n \geqslant 1} \subset \D$ with the following properties: 
\begin{enumerate}
\item $(I - \zeta T)$ is injective; 
\item $\lambda_{n}$ tends to $\zeta$ non-tangentially as $n \to \infty$.
\end{enumerate}
Let $x \in \HH$ and $p \in \N$. Then the sequence 
$$\big\{(I - \lambda_n T)^{-p} x\big\}_{n \geqslant 1}$$
is uniformly bounded if and only if $x \in \Rng(I - \zeta T)^{p}$, in which case, 
$$(I - \lambda_n T)^{-p} x \to (I - \zeta T)^{-p} x$$ weakly in $\HH$. 
\end{Lemma}

\begin{proof}[Proof of Theorem \ref{thmACbis}]
$(i) \implies (ii)$: Since the norms of the reproducing kernels $k_{\lambda, N}^{\HH}$ are the norms of the evaluation functionals $f \mapsto f^{(N)}(\lambda)$, we can apply the uniform boundedness principle to see, for fixed $\alpha > 1$, that if the $N$-th derivative 
of every function in $\HH$
has a finite limit as $\lambda \to \zeta_0$ with $\lambda \in \Gamma_{\alpha}(\zeta_0)$, 
then the norms of the kernels $k_{\lambda,N}^{\HH}$
are uniformly bounded for $\lambda \in \Gamma_{\alpha}(\zeta_0)$. 

$(ii)\implies(iii)$: By
Lemma~\ref{Lem:formule-noyau-reproduisant-derive}, the vectors 
$$(I-\overline{z_n}X_{\HH}^*)^{-(N+1)}{X_{\HH}^*}^N k_0^{\HH}$$ are uniformly bounded for any sequence $\{z_n\}_{n \geqslant 1} \subset \Gamma_{\alpha}(\zeta_0)$ tending to $\zeta_0$. 
By our assumption $\sigma_p(X_{\HH}^*)\subset \D$ (Definition \ref{Def-admissible}) we see that
the operator $I-\overline{\zeta_0}X_{\HH}^*$ is injective. Now apply Lemma \ref{h2723n2340} to conclude that ${X_{\HH}^*}^N k_0^{\HH} \in \Rng(I-\overline\zeta_0 X_{\HH}^*)^{N+1}$. 


$(iii)\implies (i)$: Again using Lemma \ref{h2723n2340}, we see that 
\[
(I-\overline{z_n}X_{\HH}^*)^{-(N+1)}{X_{\HH}^*}^N k_0^{\HH}\to (I-\overline{\zeta_0}X_{\HH}^*)^{-(N + 1)}{X_{\HH}^*}^N k_0^{\HH}
\]
weakly for any sequence $\{z_n\}_{n \geqslant 1} \subset \Gamma_{\alpha}(\zeta_0)$ tending 
to $\zeta_0$.
However, Lemma~\ref{Lem:formule-noyau-reproduisant-derive} says that the left hand side of the identity above is precisely  
$\frac{1}{N!}\,k_{z_n,N}^{\HH}$.
Hence, for any $f\in \HH$, the $N$-th derivative $f^{(N)}(z_n)$ has a finite limit as $z_n$ tends 
to $\zeta_0$ within $\Gamma_{\alpha}(\zeta_0)$. 

To see that the lower order derivatives $f, f', f'', \ldots, f^{(N - 1)}$ have finite non-tangential limits at $\zeta_0$, use an argument from the proof of Theorem 21.26 in \cite{FM}.

Finally, the equivalent conditions of the theorem show that the linear functional $f\mapsto f^{(N)}(\zeta_0)$ is continuous on $\HH$ and thus, by the Riesz representation theorem, it is induced by a kernel $k_{\zeta_0,N}^{\HH} \in \HH$ satisfying 
\[
(I-\overline{\zeta_0}X_{\HH}^*)^{-(N+1)}{X_{\HH}^*}^N k_0^{\HH}=\frac{1}{N!}\,k_{\zeta_0,N}^{\HH}.
\]
This proves \eqref{eq:noyau-reproduisant-derivee-n-ieme-range-of-operatorXbstar}.
\end{proof}

This next result helps us to produce a large class of admissible reproducing kernel Hilbert spaces. 

\begin{Lemma}\label{Fallah}
Let $\HH$ be a $B$-invariant reproducing kernel Hilbert space of analytic functions on $\D$ such that
the analytic polynomials are dense in $\HH$. Then
$
 \sigma_p(X_{\HH}^*)  = \emptyset.
$
In particular, if  $X_{\HH} = B|_{\HH}$ acts as a contraction on $\HH$, then $\HH$ is an admissible space. 
\end{Lemma}



\begin{proof}
Suppose $\lambda\in \C$ and $f \in \HH \setminus \{0\}$ with 
$X_{\HH}^*f= \lambda f$. On one hand,
$\langle X_{\HH}^*f,z^n\rangle_{\HH}=\lambda \langle f,z^n\rangle_{\HH}$, while
on the other hand,
$$\langle X_{\HH}^*f,z^n\rangle_{\HH}=\langle f,X_{\HH} z^n\rangle_{\HH} = \langle f,z^{n-1}\rangle_{\HH}, \quad n\geqslant 1.$$
Combining these two facts yields 
\begin{equation}\label{oiweurowieur}
\lambda \langle f,z^n\rangle_{\HH} = \langle f,z^{n-1}\rangle_{\HH} \quad n\geqslant 1.
\end{equation}
If $\lambda = 0$, the previous identity shows that $\langle f, z^k\rangle_{\HH} = 0$ for all $k \geqslant 0$. By the density of the polynomials in $\HH$ we see that $f = 0$ -- a contradiction.

If $\lambda \not = 0$ then 
  $$\lambda \langle f,1\rangle_{\HH}=\langle X_{\HH}^*f,1\rangle_{\HH}=\langle f,X_{\HH} 1\rangle_{\HH}=0$$  and thus 
$\langle f, 1\rangle_{\HH} = 0$.
Use this last identity and repeatedly apply \eqref{oiweurowieur} to see that
$
\langle f,z^k \rangle_{\HH}=0$ for all $k\geqslant 0.
$
Again, by our assumption that the polynomials are dense in $\HH$, we see that $f=0$. 
\end{proof}

\begin{Remark}\label{Remark-CK}
If  $\HH$ contains all of the Cauchy kernels $k_{w}$, $w \in \D$ (see \eqref{Cauchy-kernel}), then we can use the fact that 
$X_{\HH} k_{w} = \overline{w} k_{w}$
to replace the identity in \eqref{oiweurowieur} with 
$\lambda \langle{f, k_w} \rangle_{\HH} = w \langle f, k_w \rangle_{\HH}$. Thus the hypothesis ``the polynomials are dense in $\HH$''  in Lemma~\ref{Fallah} can be replaced with ``the linear span of Cauchy kernels are dense in $\HH$''. We would like to thank Omar El Fallah for some fruitful discussions concerning an earlier version of this result. 
\end{Remark}

Here are three applications of Theorem \ref{thmACbis}. 

\subsection*{$\MM(\overline{a})$-spaces} For $a \in H^{\infty}$ we want to show that $\MM(\overline{a})$ is admissible. By Proposition \ref{Prop:outer-wlog} we can assume that $a$ is outer. To verify that $\MM(\overline{a})$ is admissible, we will check the hypothesis of Lemma \ref{Fallah}. 
It is clear that $\MM(\overline{a})$ is $B$-invariant (use the identity
$T_{\overline{z}} T_{\overline{a}} = T_{\overline{a}} T_{\overline{z}}$  from Proposition \ref{Toe-facts}(3)).

To show that $B=T_{\overline{z}}$ is contractive on $\MM(\overline{a})$, notice that for any $g \in H^2$ we have 
$$\|B T_{\overline{a}} g\|_{\overline{a}} = \|T_{\overline{z}} T_{\overline{a}} g\|_{\overline{a}} =  \|T_{\overline{a}} T_{\overline{z}} g\|_{\overline{a}} = \|T_{\overline{z}} g\|_{H^2} \leqslant \|g\|_{H^2} = \|T_{\overline{a}} g\|_{\overline a}.$$
Thus $\|B\|_{\MM(\overline{a}) \to \MM(\overline{a})} \leqslant 1$.

To finish, using Lemma~\ref{Fallah} and Remark~\ref{Remark-CK}, we need to show that the Cauchy kernels $k_{\lambda}$ belong to $\MM(\overline{a})$ and have dense  linear span.  From Proposition \ref{Toe-facts}(1) we have $k_{\lambda}=T_{\overline{a}}(k_{\lambda}/\overline{a(\lambda)})\in
\MM(\overline{a})$. Furthermore, since $T_{\overline{a}}$ is a partial isometry from $H^2$ \emph{onto} $\MM(\overline{a})$, it maps a dense subset of $H^2$ onto a dense subset of $\MM(\overline{a})$. Thus the density of the linear span of $k_\lambda$, $\lambda\in\D$, in $\MM(\overline{a})$ follows from the well-known density of this span in $H^2$. 
We remark that one can also obtain the admissibility of $\MM(\overline a)$ by showing the density of the polynomials in $\MM(\overline a)$ \cite[p. 745]{FM}.

Using Theorem \ref{thmACbis}, we obtain the following explicit characterization
of the boundary behavior for $\MM(\overline{a})$. 


\begin{Corollary}\label{AC}
Let $a \in H^{\infty}$,  $\zeta_0\in\T$, and $N\in \N_0$. Then for every
$f\in\MM(\overline{a})$, the functions $f,f',f'',\ldots,f^{(N)}$ have finite non-tangential limits at $\zeta_0$ 
if and only if
\begin{eqnarray}\label{ACa}
 \int_{0}^{2 \pi} \frac{|a(e^{it})|^2}{|e^{it}-\zeta_0|^{2N+2}}dt<\infty.
\end{eqnarray}
We will write $\zeta_0\in (AC)_{\overline{a},N}$ if the condition \eqref{ACa} holds. In this case, we have 
\[
k_{\zeta_0,\ell}^{\overline{a}}=T_{\overline a}(ak_{\zeta_0,\ell}),\qquad 0\leqslant \ell\leqslant N,
\]
where 
\[
ak_{\zeta_0,\ell}= \ell! \frac{z^\ell a}{(1-\overline{\zeta_0} z)^{\ell+1}}.
\]
Moreover, for each $\alpha > 1$ we have
\[
\lim_{\substack{\lambda\to\zeta_0\\\lambda\in\Gamma_\alpha(\zeta_0)}}\|k_{\lambda,\ell}^{\overline a}-k_{\zeta_0,\ell}^{\overline a}\|_{\overline a}=0.
\]
\end{Corollary}

\begin{proof}
Corollary~\ref{RK-formula} gives us
$$
\|k_{\lambda,N}^{\overline{a}}\|_{\overline{a}}^2=(N!)^2 \int_{0}^{2\pi}\dfrac{|a(e^{it})|^2}{|e^{it}-\lambda|^{2N+2}}\,\frac{dt}{2\pi}.
$$
If $\lambda$ approaches $\zeta_0$ from within a fixed Stolz domain $\Gamma_\alpha(\zeta_0)$, then
$$
\frac{1}{|e^{it}-\lambda|}\leqslant \frac{\alpha+1}{|e^{it}-\zeta_0|}, \quad t \in [0, 2 \pi],
$$
and so 
\begin{equation}\label{z9eneyskl}
\frac{|a(e^{it})|^2}{|e^{it}-\lambda|^{2N+2}}\leqslant (\alpha+1)^{2N+2} \frac{|a(e^{it})|^2}{|e^{it}-\zeta_0|^{2N+2}}.
\end{equation}
If 
$$
\int_{0}^{2\pi}\dfrac{|a(e^{it})|^2}{|e^{it}-\zeta_0|^{2N+2}}\,\frac{dt}{2\pi}<\infty
$$
we see that 
\begin{equation}\label{eq:norme-borne-stotlz}
\sup\{\|k_{\lambda,N}^{\overline{a}}\|_{\overline{a}}:\lambda\in \Gamma_{\alpha}(\zeta_0)\}<\infty. 
\end{equation}
Now apply Theorem \ref{thmACbis}.

Conversely, if for every $f\in\MM(\overline{a})$, the functions $f,f',f'',\ldots,f^{(N)}$ have non-tangential limits at $\zeta_0$, then Theorem~\ref{thmACbis} implies that for each fixed $\alpha>1$, the condition \eqref{eq:norme-borne-stotlz} is satisfied. Thus  
\[
\sup_{\lambda\in\Gamma_{\alpha}(\zeta_0)}\int_0^{2\pi}\frac{|a(e^{it})|^2}{|e^{it}-\lambda|^{2N+2}}\frac{dt}{2\pi}<\infty.
\]
By Fatou's Lemma 
\[
\int_{0}^{2\pi}\dfrac{|a(e^{it})|^2}{|e^{it}-\zeta_0|^{2N+2}}\,\frac{dt}{2\pi}\leqslant \liminf_{\substack{\lambda\to\zeta_0\\\lambda\in\Gamma_\alpha(\zeta_0)}} \int_0^{2\pi}\frac{|a(e^{it})|^2}{|e^{it}-\lambda|^{2N+2}}\frac{dt}{2\pi}<\infty.
\]

%
%
Now let $\zeta_0\in (AC)_{\overline{a},N}$. Then, for any  $f=T_{\overline a}g\in\MM(\bar a)$ and $0\leqslant \ell\leqslant N$, we have
\[
\langle f,T_{\overline{a}}(ak_{\zeta_0,\ell}) \rangle_{\overline a}=\langle g,ak_{\zeta_0,\ell}\rangle_{H^2}.
\]
Note that $ak_{\lambda,\ell}\to ak_{\zeta_0,\ell}$ in $H^2$ as $\lambda \in \zeta_0$ from within $\Gamma_{\alpha}(\zeta_0)$. Indeed this is true pointwise and, by using the inequality in \eqref{z9eneyskl} and the dominated convergence theorem, we also have $$\|ak_{\lambda,\ell}\|_{H^2}\to \|ak_{\zeta_0,\ell}\|_{H^2}$$
as $\lambda \to \zeta_0$ from within $\Gamma_{\alpha}(\zeta_0)$. By a standard Hilbert space argument we have
\begin{equation}\label{kerconv}
\|ak_{\lambda,\ell} - ak_{\zeta_0,\ell}\|_{H^2} \to 0.
\end{equation}

The above analysis says that
\begin{align*}
\langle f,T_{\overline{a}}(ak_{\zeta_0,\ell}) \rangle_{\overline a} & = \lim_{\substack{\lambda\to\zeta_0\\\lambda\in\Gamma_\alpha(\zeta_0)}} \langle g,ak_{\lambda,\ell}\rangle_{H^2}\\
& = \lim_{\substack{\lambda\to\zeta_0\\\lambda\in\Gamma_\alpha(\zeta_0)}}\langle f,T_{\overline a}ak_{\lambda,\ell}\rangle_{\overline a}.
\end{align*}
By Corollary~\ref{RK-formula}, $T_{\overline a}(ak_{\lambda,\ell})=k_{\lambda,\ell}^{\overline a}$, whence
\begin{align*}
\langle f,T_{\overline{a}}(ak_{\zeta_0,\ell}) \rangle_{\overline a} & =\lim_{\substack{\lambda\to\zeta_0\\\lambda\in\Gamma_\alpha(\zeta_0)}}\langle f,k_{\lambda,\ell}^{\overline a}\rangle_{\overline a}\\
& =\lim_{\substack{\lambda\to\zeta_0\\\lambda\in\Gamma_\alpha(\zeta_0)}}f^{(\ell)}(\lambda)\\
& = f^{(\ell)}(\zeta_0)\\
& =\langle f,k_{\zeta_0,\ell}^{\overline a}\rangle_{\overline a},
\end{align*}
which proves that $k_{\zeta_0,\ell}^{\overline{a}}=T_{\overline a}(ak_{\zeta_0,\ell})$.
Finally, from \eqref{kerconv}
$$
\|k_{\lambda,\ell}^{\overline a}-k_{\zeta_0,\ell}^{\overline a}\|_{\overline a}=\|ak_{\lambda,\ell}-ak_{\zeta_0,\ell}\|_{H^2}
 \to 0, \quad \lambda \to \zeta_0, \lambda \in \Gamma_{\alpha}(\zeta_0).\qedhere$$
\end{proof}

\begin{Remark}\label{rem:Taylor}
\hfill
\begin{enumerate}
\item In a general admissible space $\HH$ we see that if 
$$\sup\{ \|k_{\lambda, N}^{\HH}\|_{\HH}: \lambda \in \Gamma_{\alpha}(\zeta_0)\} < \infty$$
for each $\alpha > 1$, then 
$$k_{\lambda, N}^{\HH} \to k_{\zeta_0, N}^{\HH}$$
weakly in $\HH$ as $\lambda \to \zeta_0$ non-tangentially. However, it is not immediately clear if we also have norm convergence of the kernels. Corollary \ref{AC} shows this is true when $\HH = \MM(\overline{a})$. See also \cite{FM} where this was shown to be true when $\HH$ is one of the de Branges-Rovnyak spaces $\HH(b)$. 
\item The condition \eqref{ACa} yields an estimate on the rate of decrease of the outer function $a$, along with its derivatives, at the distinguished point $\zeta_0$. Indeed, using the facts that $(\zeta_0-z)^{N+1}$ is an outer function, 
along with the condition \eqref{ACa}, and Smirnov's theorem \cite{Duren} (if the boundary function of an outer function belongs to $L^2$ then the function belongs to $H^2$), the function 
$$
h(z):= \frac{a(z)}{(z-\zeta_0)^{N+1}}
$$
belongs to $H^2$. Recall the following standard estimates for the derivatives of $h \in H^2$:
$$
|h^{(\ell)}(r\zeta)|=o((1-r)^{-\ell-\frac{1}{2}}),\quad r\to 1^{-}.
$$
Thus 
Leibniz' formula yields 
\begin{align*}
a^{(k)}(r\zeta_0) & =\sum_{\ell=0}^k\binom{k}{\ell}h^{(\ell)}(r\zeta_0)\frac{d^{k-\ell}}{dz^{k-\ell}}
 (z-\zeta_0)^{N+1}\Big|_{z=r\zeta_0}\\
 & = o((1-r)^{N+\frac{1}{2}-\ell}), \quad r\to 1^{-}.
\end{align*}
In particular, we see that the functions $a,a',\dots,a^{(N)}$ have radial (and even non-tangential) limits 
$a^{(\ell)}(\zeta_0)$ which vanish for each $0\leqslant \ell\leqslant N$. 
\end{enumerate}
\end{Remark}

Corollary~\ref{AC} yields the following interesting observation which shows a sharp difference between
$\MM(\overline{a})$ spaces and the model, or more generally, de Branges-Rovnyak spaces $\HH(b)$. More precisely, when $\log (1 - |b|) \not \in L^1$, it is sometimes the case that every function in $\HH(b)$ 
can be analytically continued to an open neighborhood of a point $\zeta_{0} \in \T$. For example, if $b$ is an inner function and $\zeta_0 \in \T$ with
$$\liminf_{\lambda \to \zeta_0} |b(\lambda)| > 0,$$
then every $f \in \HH(b)$ (which turns out to be a model space $K_b$) can be analytically continued to some open neighborhood $\Omega_{\zeta_0}$ of $\zeta_0$ (see \cite[Cor. 3.1.8]{DSS} for details). 
This phenomenon never happens in $\MM(\overline{a})$. 

\begin{Proposition}\label{99999}
There is no point $\zeta_0\in\T$ such that every $f\in \MM(\overline{a})$ can be analytically
continued to some open neighborhood of $\zeta_0$. 
\end{Proposition}

\begin{proof}
Suppose there exists such a $\zeta_0\in \T$ where every function in $\MM(\overline{a})$ has an 
analytic continuation to an open neighborhood $\Omega_{\zeta_0}$ of $\zeta_0$.
Then the function $a\in\MM(a)\subset\MM(\overline{a})$ would have an analytic continuation to $\Omega_{\zeta_0}$ and thus could be expanded in a power series 
around $\zeta_0$. If
every function in $\MM(\overline a)$ had an analytic continuation to $\Omega_{\zeta_0}$, then every function in 
$\MM(\overline a)$, and its derivatives of all orders, would have finite non-tangential limits at $\zeta_0$. 
In particular, the condition \eqref{ACa} would hold for every $N\in \N$ at $\zeta_0$. By Remark~\ref{rem:Taylor}, this 
would imply that all of
the Taylor coefficients of $a$ at $\zeta_0$ would vanish, implying $a \equiv 0$ on $\D$.
\end{proof}

\subsection*{$\HH(b)$-spaces} 
We have seen that $\HH(b)$ spaces are special cases of $\MM(A)$-spaces. It turns out that they are admissible. Indeed, they are $B$-invariant reproducing kernel Hilbert spaces contained in 
$H^2$ with $\|B\|_{\HH(b) \to \HH(b)} \leqslant 1$ \cite[Theorem 18.13]{FM}. Furthermore, $\sigma_{p}(X_{\HH}^{*}) \subset \D$ \cite[Theorem 18.26]{FM}.
Thus Theorem \ref{thmACbis} applies, allowing us to reproduce some of the results of
\cite{FM08}. In particular, the condition that every $f \in \HH(b)$, along with $f', \ldots, f^{(N)}$,  has a non-tangential limit at $\zeta_0$ is equivalent to the condition that the norm of the reproducing kernels for $\HH(b)$ are uniformly bounded in every Stolz domain anchored at $\zeta_0$. The difficult part of \cite{FM08} is to prove that the boundedness of the kernels is equivalent to the condition in \eqref{FM-cont}. 

\begin{Remark}
As already mentioned in Section 3, if $a\in H_1^\infty$ is such that $\log(1-|a|)\in L^1$ and $b$ is its (outer) Pythagorean mate, then we have $\MM(\overline{a})\subset\HH(b)$.  If $N\in\N_0$ and $\zeta_0\in\T$ are such that for every $f\in\HH(b)$, the functions $f,f',\dots,f^{(N)}$ admit a finite non-tangential limit at $\zeta_0$, then this is also true for every function $f\in\MM(\overline{a})$. What is more surprising here is that the converse is true. This is a byproduct of Corollary~\ref{AC} and \cite[Theorem 3.2]{MR2390675}. Indeed, since $|b|^2=1-|a|^2$ almost everywhere on $\T$, we  see (remembering $b$ is outer) that condition \eqref{FM-cont} implies 
\[
\int_\T \frac{|\log(1-|a(\zeta)|^2)|}{|\zeta-\zeta_0|^{2N+2}}\,dm(\zeta)<\infty,
\]
which is equivalent to 
$$
 \int_{0}^{2 \pi} \frac{|a(e^{it})|^2}{|e^{it}-\zeta_0|^{2N+2}}dt<\infty.
 $$
Thus the conditions in \eqref{FM-cont} and \eqref{ACa} are equivalent which shows that the existence of boundary derivatives for functions in $\HH(b)$ and $\MM(\overline{a})$ (in the case when $b$ is outer) are equivalent. \end{Remark}

\subsection*{$\DD(\mu)$-spaces} For a positive finite Borel measure $\mu$ on $\T$ let 
$$\varphi_{\mu}(z) = \int_{\T} \frac{1 - |z|^2}{|\xi - z|^2} d\mu(\xi), \quad z \in \D,$$
be the Poisson integral of $\mu$. The {\em harmonically weighted Dirichlet space} $\mathscr{D}(\mu)$  \cite{KEFR, Ri} is the set of all $f \in \mathscr{O}(\D)$ for which 
$$\int_{\D} |f'|^2 \varphi_{\mu} dA < \infty,$$
where $d A  = dx dy/\pi$ is normalized planar measure on $\D$. Notice that when $\mu$ is Lebesgue measure on $\T$, then $\varphi_{\mu} \equiv 1$ and $\mathscr{D}(\mu)$ becomes the classical Dirichlet space \cite{KEFR}. One can show that $\mathscr{D}(\mu) \subset H^2$ \cite[Lemma 3.1]{Ri} and the norm $\|\cdot\|_{\DD(\mu)}$ satisfying 
$$\|f\|_{\mathscr{D}(\mu)}^2 := \|f\|^{2}_{H^2} + \int_{\D} |f'|^2 \varphi_{\mu} dA$$ makes $\mathscr{D}(\mu)$ into a reproducing kernel Hilbert space of analytic functions on $\D$. It is known that both the polynomials and the linear span of the Cauchy kernels form a dense subset of $\mathscr{D}(\mu)$ \cite[Corollary 3.8]{Ri}.

The backward shift  $B$ is a well-defined contraction on $\DD(\mu)$. Indeed, we have
$$\|z f\|_{\DD(\mu)} \geqslant \|f\|_{\DD(\mu)}, \quad f \in \DD(\mu),$$ and the constant function $1$ is orthogonal to $z \DD(\mu)$ \cite[Theorem 3.6]{Ri}. Thus 
$$
\|f\|_{\DD(\mu)}^2= \|f(0)+zBf\|_{\DD(\mu)}^2=|f(0)|^2+\|z B f\|_{\DD(\mu)}^2 \geqslant \|B f\|_{\DD(\mu)}^2.$$
We thank Stefan Richter for showing us this elegant argument. 
From Lemma \ref{Fallah} we see that $\DD(\mu)$ is an admissible space. 

Using a kernel function estimate from \cite{EEK}, 
one can show that if 
$$\mu = \sum_{1 \leqslant j \leqslant n} c_j \delta_{\zeta_j}, \quad c_j > 0, \zeta_j \in \T,$$
then each of the kernels 
$$k_{r \zeta_j}^{\DD(\mu)}, \quad 1 \leqslant j \leqslant n,$$ remains norm bounded as $r \to 1^{-}$. Thus the radial limits of every function from $\DD(\mu)$ exist at each of the $\zeta_j$. Other radial limit results along these lines can be stated in terms of an associated capacity for $\DD(\mu)$ \cite{EEK, MR3000683}.


\section{An orthogonal decomposition}\label{section5}

The goal of this last section is to determine, whenever it exists, the orthogonal complement of $\MM(a)$ in $\MM(\overline{a})$. We begin our discussion with a few interesting and representative examples. 

\begin{Example}\label{ppappap}
If $I$ is inner, then $a := 1 + I$ is outer. 
Moreover, one can quickly verify that
$$\frac{a}{\overline{a}} = I \quad \mbox{a.e.\! on $\T$}.$$ Since $I H^2$ is a closed subspace of $H^2$ (multiplication by an inner function is an isometry on $H^2$), we see that $T_{a/\overline{a}} = T_I$ has closed range. Hence, 
\begin{equation}\label{kkksjjs}
H^2 = T_{a/\overline{a}} H^2 \oplus_{H^2} (H^2 \ominus_{H^2} T_{a/\overline{a}} H^2) = T_{a/\overline{a}} H^2 \oplus_{H^2} \ker{T_{\overline{a}/a}}.
\end{equation}
Since $a$ is outer, then $T_{\overline{a}}$ is injective (Proposition \ref{Toe-facts}(2)) and so by \eqref{bbdggdyyd}, $T_{\overline{a}}$ is an isometry from $H^2$ onto $\MM(\overline{a})$. Applying $T_{\overline{a}}$ to both sides of \eqref{kkksjjs} and using the  earlier mentioned operator identity 
$$T_{\overline{a}} T_{a/\overline{a}} = T_{a}$$ 
(Proposition \ref{Toe-facts}(3)), we obtain
$$\MM(\overline{a}) = \MM(a) \oplus_{\overline{a}} T_{\overline{a}} \ker{T_{\overline{a}/a}}.$$
Now bring in the identity $T_{\overline{a}/a} = T_{\overline{I}}$ and the facts that 
$\ker{T_{\overline{I}}}= K_I$ (Proposition \ref{Toe-facts}(4)) and $T_{\overline{a}} K_I = K_I$ (to see this last fact, observe that $T_{\overline{a}} K_I \subset K_I$ -- Proposition \ref{Toe-facts}(2) -- and if $f \in K_I$ then $T_{\overline{a}} f = f + T_{\overline{I}} f = f$ and so $T_{\overline{a}} K_I = K_I$), to finally obtain the orthogonal decomposition 
$$\MM(\overline{a}) = \MM(a) \oplus_{\overline{a}} K_I.$$
\end{Example}

\begin{Example}\label{866wywhwnnw}
For the outer function
$$
a :=\prod_{1 \leqslant j \leqslant n}(z-\zeta_j)^{m_j}, \quad \zeta_j \in \T, m_j \in \N,$$
 one can verify that 
$$\frac{a}{\overline{a}} = c I \quad \mbox{on $\T$},$$
where 
$$I(z) = z^N, \quad |c| = 1, \quad N=\sum_{1 \leqslant j \leqslant n} m_j.$$ The same analysis as in the previous example shows that 
$$\MM(\overline{a}) = \MM(a) \oplus_{\overline{a}} T_{\overline{a}} K_I.$$
Now observe that $K_I = \mathcal{P}_{N - 1}$ (the analytic polynomials of degree at most $N - 1$) and 
$T_{\overline{a}} \mathcal{P}_{N - 1} = \mathcal{P}_{N - 1}$. Indeed $T_{\overline{a}} \mathcal{P}_{N - 1} \subset \mathcal{P}_{N - 1}$ and equality follows since $\mathcal{P}_{N - 1}$ is finite dimensional
and $T_{\overline{a}}$ is injective. Thus we get 
 $$\MM(\overline{a})=\MM(a)\oplus_{\overline{a}} \mathcal{P}_{N-1}.$$
\end{Example}

\begin{Example}
Suppose $I$ is inner, $m \in \N$, and 
$$a := (1 - I)^{m}.$$ Again, as we have seen in the previous two examples, 
$$\frac{a}{\overline{a}} = c I^{m} \quad \mbox{a.e.\! on $\T$,}$$ for some suitable unimodular constant $c$, and so 
$$\MM(\overline{a}) = \MM(a) \oplus_{\overline{a}} T_{\overline{a}} K_{I^m}.$$ Here things are a bit more tricky since it is not as clear as it was before that $T_{\overline{a}} K_{I^m}  = K_{I^m}$. However, by applying the following technical lemma, this is indeed the case. 
\end{Example}

\begin{Lemma}\label{ThmInt}
Let $a\in H^{\infty}$ be outer and $I$ inner.
Then the following are equivalent:
\begin{enumerate}
\item[(i)]
 $T_{\overline{a}}K_I=K_I$;
\item[(ii)] There exists a $\psi \in H^{\infty}$ such that 
$
 a\psi-1\in IH^{\infty}
$;
\item[(iii)] There exists a constant $\delta>0$ such that 
$|a| + |I| \geqslant \delta$ on $\D$.
\end{enumerate}
\end{Lemma}

\begin{proof}
$(i) \iff (ii)$: As we have already seen, since $a$ is outer, $T_{\overline{a}}$ is injective on $H^2$ and hence on 
$K_I$. In order to have $T_{\overline{a}}K_I=K_I$, the operator $T_{\overline{a}}$ must be invertible on $K_I$. This is equivalent to saying that the compression of the analytic Toeplitz operator $T_a$ to $K_I$ (a truncated Toeplitz operator), i.e., 
$$a(S_I) :=P_I T_a|_{K_I}$$
(where $P_I$ is the orthogonal projection of $L^2$ onto $K_I$,
$S_{I} := P_I T_z|_{K_I}$ is the compression of the shift $T_z$, and $a(S_I)$ is defined via the functional calculus) is invertible on $K_I$. If $a(S_I)$ is
invertible then its inverse commutes with $S_I$ \cite[p. 231]{Nik1}. By the commutant lifting theorem, there is a $\psi\in H^{\infty}$ such that 
$$(a(S_I))^{-1}
=\psi(S_I)$$ and thus for every $f\in K_I$, $P_I(a\psi f)=f$, or equivalently, 
$(a\psi-1)f\in IH^2$. This translates to the condition
$a\psi-1\in IH^{\infty}$ (pick for instance $f = 1 - \overline{I(0)} I$ which is
outer with bounded recicprocal). Clearly, when $a\psi-1\in IH^{\infty}$, we can reverse the argument.

The equivalence $(2)\Longleftrightarrow (3)$ is an application of the corona theorem \cite{Garnett}. 
\end{proof}

\begin{Example}\label{oooiu}
Let 
\[
 a :=\prod_{1 \leqslant j \leqslant n}(\zeta_j-I_j)^{m_j},
\]
where $I_j$ are inner functions, $\zeta_j \in \T$, and $m_j \in \N$.

As with the previous two examples, 
$$
\frac{a}{\overline{a}} = c I \quad \mbox{a.e.\!  on $\T$},$$ 
where 
 $$I =\prod_{1 \leqslant j \leqslant n}I_j^{m_j}, \quad |c| = 1.$$
Hence 
\[
 \MM(\overline{a})=\MM(a)\oplus_{\overline{a}}T_{\overline{a}}(K_I).
\]
Here things become more complicated than in our previous examples since, as we will see shortly, $T_{\overline{a}} K_I$ can be a proper subspace of $K_I$ that is difficult to identify. Note however that since $a$ is outer then one can easily prove that $T_{\overline{a}} K_I$ is dense in $K_I$ (in the $H^2$ norm). 

For example, if 
$$a :=(1-I_1)(1-I_2), \quad I=I_1I_2,$$ then $(a,I)$ is
not always a corona pair and so, by Lemma \ref{ThmInt}, $T_{\overline{a}} K_I$ is a proper subspace of $K_I$. 

More precisely, let 
$$\lambda_n=1-4^{-n^2}, \quad \Lambda_1=(\lambda_n)_{n\geqslant 1},$$
$$\mu_n=1-4^{-n^2-n}, \quad \Lambda_2=(\mu_n)_{n\geqslant 1},$$ $I_1$ and $I_2$ the Blaschke products with these zeros. 
In order to show that 
$$\inf \{|a(z)| + |I(z)|: z \in \D\} = 0,$$ it is enough to show that $I_1(\mu_{n_k})\to 1$ when $k\to\infty$ for some suitable
sub-sequence $(\mu_{n_k})$.
Clearly $I_1(\mu_n)$ is a real number. Since the zeros of $I_1$ are simple, $I_1$ changes
sign on $[0,1)$ at each $\lambda_n$. We can thus assume that for alternating $\mu_n$,  we
have $I_1(\mu_n)>0$. Note these $\mu_n$ by $\mu_n^+$. 
Finally, since the sequence is interpolating with increasing pseudohyperbolic
distances between successive points, we necessarily have $I_1(\mu_n^+)\to 1$.
Hence $$a(\mu_n^+)=(1-I_1(\mu_n^+))(1-I_2(\mu_n^+))\to 0, \quad n\to\infty,$$ and 
$I(\mu_n^+)=0$, which proves the claim.

\end{Example}

In general, we see from the discussion in our first example that if $a \in H^{\infty}$ is outer and $\MM(a)$ is a closed subspace of $\MM(\overline{a})$ (and this is not always the case), then, as we will explain why in a moment,  
$$\MM(\overline{a}) = \MM(a) \oplus_{\overline{a}} T_{\overline{a}} \ker{T_{\overline{a}/a}}.$$
So the issues we need to discuss further are:
\begin{enumerate}
\item When is $\MM(a)$ a closed subspace of $\MM(\overline{a})$?
\item Identify $\ker{T_{\overline{a}/a}}$.
\item Identify $T_{\overline{a}} \ker{T_{\overline{a}/a}}$.
\end{enumerate}

In order to avoid trivialities, we point out the following:.

\begin{Proposition}
Let $a \in H^{\infty}$ and outer. 
\begin{enumerate}
\item If $T_{\overline{a}}$ is surjective, then $\MM(a) = \MM(\overline{a}) = H^2$.
\item $\MM(a)=\MM(\overline{a})$ if and only if $T_{a/\overline{a}}$ is surjective. 
\end{enumerate}
\end{Proposition}

\begin{proof}
(1): From Proposition \ref{Toe-facts}(2) we know that $T_{\overline{a}}$ is injective. Thus if $T_{\overline{a}}$ were surjective it would also be invertible (as would $T_{a}$). Hence $\MM(a)=\MM(\overline{a})=H^2$. 

(2): Note that 
\begin{equation}\label{z,xjcn,zxnc}
\MM({a})=T_{\overline{a}}T_{a/\overline{a}}H^2,
\end{equation}
and since $T_{\overline{a}}$ is injective, we get that 
$$\MM(a)=\MM(\overline{a}) \iff T_{a/\overline{a}}H^2=H^2. \qedhere$$
\end{proof}

From now on we will assume that $T_{a/\overline{a}}$ is not surjective. This next result helps us determine when $\MM(a)$ is closed in $\MM(\overline{a})$. 

\begin{Proposition}\label{04983}
For $a \in H^{\infty}$ and outer, the following are equivalent:
\begin{enumerate}
\item[(i)] $\MM(a)$ is a closed subspace of $\MM(\overline{a})$.
\item[(ii)] $T_{a/\overline{a}} H^2$ is a closed subspace of $H^2$.
\item[(iii)] $T_{a/\overline{a}}$ is left invertible.
\item[(iv)] $T_{\overline{a}/a}$ is surjective. 
\end{enumerate}
\end{Proposition}

\begin{proof} 
Using \eqref{z,xjcn,zxnc} and the fact that $T_{\overline{a}}$ is an isometry from $H^2$ onto $\MM(\overline{a})$, 
we see that $\MM(a)$ is  a closed subspace of $\MM(\overline{a})$ if and only
if $T_{a/\overline{a}}H^2$ is a closed subspace of $H^2$. This proves $(i) \iff (ii)$. 
For the remaining implications, use the fact that $T_{a/\overline{a}}$ is  injective (Proposition \ref{Toe-facts}(2)) along with the general fact that for a bounded linear operator $A$ on a Hilbert space, the conditions $A$ is left invertible; $A$ is injective with closed range;  $A^{*}$ is surjective -- are equivalent.
%
%
\end{proof}

When $\MM(a)$ is a closed subspace of $\MM(\overline{a})$ then $T_{a/\overline{a}}$ has closed range and so, by using the analysis from Example \ref{ppappap}, 
$$\MM(\overline{a}) = \MM(a) \oplus_{\overline{a}} T_{\overline{a}} \ker{T_{\overline{a}/a}}.$$
This brings us to some of the subtleties of $\ker{T_{\overline{a}/a}}$ discussed earlier.  Note that $\ker{T_{\overline{a}/a} }\not = \{0\}$ since $T_{a/\overline{a}}$ is not surjective but left invertible. Recall from Theorem \ref{Hitt-Thm} and the discussion thereafter that 
$$\ker{T_{\overline{a}/a}} = \gamma K_I,$$
where 
$$\gamma = \frac{\alpha}{1 - \beta_0 I}$$
and $\alpha \in H^{\infty}_{1}$ and outer, $\beta_0$ is a Pythagorean mate, and $I$ is an inner function with $I(0) = 0$. As a consequence of Proposition \ref{04983} and Theorem \ref{7796316}, we see that 
$T_{a/\overline{a}}$ has closed range if and only if $|\gamma_0|^2$ is an $(A_2)$ weight, where
\[
 \gamma_0 =\frac{\alpha}{1-\beta_0}.
\]
Thus $\MM(a)$ is a closed non-trivial subspace of $\MM(\overline{a})$ if and only if $|\gamma_0|^2$ is an $(A_2)$ weight. We summarize this discussion with the following: 


\begin{Theorem}\label{MainThm2}
Let $a \in H^{\infty}$ be outer. Then 
\begin{enumerate}
\item $\MM(a)$ is a closed subspace of $\MM(\overline{a})$ if and only if $|\gamma_0|^2$ is an $(A_2)$ weight. 
\item If $\gamma$ and $I$ are the associated functions as above,
then
\begin{equation}\label{38emdeiem}
 \MM(\overline{a})=\MM(a)\oplus_{\overline{a}} T_{\overline{a}}(\gamma K_I).
\end{equation}

\end{enumerate} 
\end{Theorem}

Although Theorem \ref{MainThm2} might appear implicit, it actually yields a recipe to construct further, more subtle, decompositions. For example, choose an outer $\alpha \in H^{\infty}_{1}$ such that its 
Pythagorean mate $\beta_0$ satisfies the property that $|\gamma_0|^2$ is an $(A_2)$ weight. We will see a specific example of this in a moment. 
As mentioned earlier, the $(A_2)$ condition implies that $\gamma_{0}^{2}$ is a rigid function. Let $I$ be any inner function with $I(0) = 0$ and $\gamma=\alpha/(1-I \beta_0)$. From \eqref{kerker}
we have
$\gamma K_I = \ker{T_{\overline{I \gamma}/\gamma}}$.
Set 
$$a=(1+I)\gamma.$$
Then 
$$\frac{\overline{a}}{a} = \frac{\overline{I \gamma}}{\gamma} \quad \mbox{a.e.\! on $\T$}$$
and so 
$$\ker{T_{\overline{a}/a}} = \gamma K_I,$$
 whence 
\begin{equation}\label{vzzvvvz}
 \MM(\overline{a})=\MM(a)\oplus_{\overline{a}} T_{\overline{a}}(\gamma K_I).
\end{equation}

Here is an example which uses this recipe. 

\begin{Example}
Let $\varepsilon \in (0, \tfrac{1}{2})$ and define the outer function $\alpha \in H^{\infty}_{1}$ by 
$$\alpha(z) = \left(\frac{1 - z}{2}\right)^{\varepsilon}.$$ With $\beta_0$ the outer Pythagorean mate for $\alpha$, an estimate from \cite[p. 359-360]{HS} yields 
$$|1 - \beta_0(\zeta)| \asymp |1 - \zeta|^{2 \varepsilon}, \quad \zeta \in \T.$$ The function $\gamma_0 = \alpha/(1 - \beta_0)$ satisfies 
$$|\gamma_{0}(\zeta)| \asymp |1 - \zeta|^{-\varepsilon}, \quad \zeta \in \T.$$ A routine estimate will show that the condition \eqref{A2} holds and so $|\gamma_0|^{2}$ is an $(A_2)$ weight. For any inner $I$ with $I(0) = 0$, define $\gamma = \alpha/(1 - I \beta_0)$ and $a = \gamma (1 + I)$ and follow the above argument to obtain the decomposition in  \eqref{vzzvvvz}.
\end{Example}

It is also possible to start from $\gamma_0(z)=(1-z)^{\varepsilon}$. Then $\beta_0$ can be 
expressed using the integral representation \eqref{beta} and $\alpha=\gamma_0(1-\beta_0)$.



We now produce a formula for the orthogonal projection $P$ from $\MM(\overline{a})$ onto $T_{\overline{a}}(\gamma K_I)$. 

\begin{Theorem}
In the above notation, let $P_I$ denote the orthogonal projection of $H^2$ onto $K_I$. Then 
$
 P=T_{\overline{a}}\gamma P_I\overline{\gamma}T_{1/\overline{a}}
$
is the orthogonal projection from $\MM(\overline{a})$ onto
$T_{\overline{a}}(\gamma K_I)$.
\end{Theorem}

\begin{proof} From Theorem \ref{Hitt-Thm} we know that $\gamma$ is an
isometric multiplier of $K_I$. The operator $P_0 :=\gamma P_I \overline{\gamma}$ is the orthogonal projection from 
$H^2$ onto $\gamma K_I$. Indeed, it is clear that its range is $\gamma K_I$. From Theorem \ref{Hitt-Thm} we deduce that $P_0(\gamma f)=\gamma f$, when $f\in K_I$. Finally it
is straight forward to see that $P_0 f=0$ whenever $f \perp \gamma K_I$. Since $T_{\overline{a}}$ 
is a an isometric isomorphism from $H^2$ onto $\MM(\overline{a})$, we can
define its inverse, which is just $T_{1/\overline{a}}$. The result now follows by composition. 
\end{proof}

To help us better understand some of the contents of $T_{\overline{a}}(\gamma K_I)$ we have the following:

\begin{Proposition}\label{Prop:kernel-boundary-orthogonalcomplement}
If $\zeta_0\in (AC)_{\overline{a},N}$, then
$$k_{\zeta_0,\ell}^{\overline{a}}\in T_{\overline{a}}(\gamma K_I), \quad  \quad 0 \leqslant \ell \leqslant N.$$
\end{Proposition}

\begin{proof}
Notice that 
$$\zeta_0\in (AC)_{\overline{a},\ell} \implies \zeta_0\in (AC)_{\overline{a},\ell'}, \quad 0 \leqslant \ell'\leqslant \ell,$$ and so it suffices to prove the result when $\ell=N$. By Theorem~\ref{MainThm2}, we can do this by proving 
$$k_{\zeta_0,N}^{\overline{a}} \perp_{\overline{a}} aH^2.$$
To prove this last fact, set $f=ah$, where $h\in H^2$. By Leibniz's formula, 
$$
\langle f,k_{r\zeta_0,N}^{\overline{a}}\rangle_{\overline{a}}=f^{(N)}(r\zeta_0)=\sum_{0 \leqslant k \leqslant N} \binom{N}{k}a^{(k)}(r\zeta_0)h^{(N-k)}(r\zeta_0).
$$
But according to Remark~\ref{rem:Taylor} we have
$$
|a^{(k)}(r\zeta_0)h^{(N-k)}(r\zeta_0)|=o((1-r)^{N+\frac{1}{2}-k}(1-r)^{k-N-\frac{1}{2}})=o(1).
$$
Thus 
$$\lim_{r\to 1}\langle f,k_{r\zeta_0,N}^{\overline{a}}\rangle_{\overline{a}}=0,$$ and, using Corollary~\ref{AC}, yields
$$
\langle f,k_{\zeta_0,N}^{\overline{a}}\rangle_{\overline{a}}=0.
$$
This proves the result. 
\end{proof}

Using Proposition~\ref{Prop:kernel-boundary-orthogonalcomplement}, we can revisit Example \ref{866wywhwnnw} and give an alternate description of the orthogonal complement of $\MM(a)$ in $\MM(\overline{a})$ when 
$$a =\prod_{1 \leqslant j \leqslant n} (z-\zeta_j)^{m_j}.$$ Indeed, since $a$ is a polynomial, it is clear that $\zeta_j\in (AC)_{\overline{a},m_j-1}$, and so $$k_{\zeta_j,\ell}^{\overline{a}}\in T_{\overline{a}}(\gamma K_I)=\mathcal P_{N-1}, \quad 1\leqslant j\leqslant n, 0\leqslant \ell\leqslant m_j-1.$$ Since the functions
$$\{k^{\overline{a}}_{\zeta_j,\ell}: j=1,\cdots, n, \ell=0,\cdots,m_i-1\}$$ are linearly independent, we obtain
\[
 \mathcal{P}_{N-1}=\bigvee\{k^{\overline{a}}_{\zeta_j,\ell}:j=1,\cdots, n,\ell=0,\cdots,m_i=1\}.
\]

\begin{Corollary}\label{kernelspan}
If $a =\prod_{j=1}^n(z-\zeta_j)^{m_j}$, then
\[
 \MM(\overline{a})=\MM(a)\oplus_{\overline{a}}
 \bigvee\{k^{\overline{a}}_{\zeta_j,\ell}:j=1,\cdots, n,\ell=0,\cdots,m_i=1\}.
\]
\end{Corollary}

The techniques above also give the following which generalizes a result from \cite{FHR, LanNow}.

\begin{Theorem}\label{Thm:Hb-decomposition-orthogonale}
Let $I$ be any inner function vanishing at 0, set $a=(1-I)/2$ and $b=(1+I)/2$.  Then
\[
\HH(b)=\MM(a)\oplus_b K_I.
\]
\end{Theorem}

\begin{proof}
From \cite{Sa} we know that $(a,b)$ forms a corona pair (see \eqref{CP}), whence $\HH(b)=\MM(\overline a)$. It follows from our first example of this section that 
we can decompose $\HH(b)$ as the direct sum of $\MM(a)$ and $K_I$
with respect to $\langle \cdot,\cdot\rangle_{\overline{a}}$. 
It remains to prove that $\MM(a)$ and $K_I$ are orthogonal in the inner product of $\HH(b)$. In other words, we need to check that given any function $f\in H^2$ and any function $g\in K_I$, we have 
\[
\langle af,g\rangle_b=0. 
\]
Using again that $\ker T_{\overline{I}}=K_I$ so that
$g=2T_{\overline a}g$, and using a well-known formula for the inner product in $\HH(b)$ \cite{Sa}, we have
\[
\langle af,g/2\rangle_b=\langle af,T_{\overline a}g\rangle_{H^2}+\langle T_{\overline b}(af),T_{\overline b}T_{\overline a}g\rangle_{\HH(\bar b)}.
\]
Note that 
\[
T_{\overline b}(af)=T_{\overline a} T_{a/\overline a}(\overline b f)\quad\mbox{and}\quad T_{\overline b}T_{\overline a}g=T_{\overline a}T_{\overline b}g.
\]
Since $\HH(\overline b)$ and $\MM(\overline a)$ coincide as Hilbert spaces, we deduce that 
\[
\langle af,g/2\rangle_b=\langle af,T_{\overline a}g\rangle_{H^2}+\langle T_{a/\overline a}(\overline b f),T_{\overline b}g\rangle_{H^2}.
\] 
Note that $T_{\overline{a}}g=\frac{1}{2}g=T_{\overline{b}}g$ and $a+b=1$. Hence
\begin{align*}
\langle af,g\rangle_b
 &=\langle af,g\rangle_{H^2}+\langle T_{a/\overline a}(\overline b f),g\rangle_{H^2}\\
  &=\langle af,g\rangle_{H^2}+\langle af,\frac{b}{a}g\rangle_{H^2}\\
 &=\langle af,g+\frac{b}{a}g\rangle_{H^2}\\
 &=\langle af,\frac{1}{a}g\rangle_{H^2}\\
 & =\langle T_{a/\overline{a}}f,g\rangle_{H^2}.
\end{align*}
Recall that $T_{a/\overline{a}} H^2$ is a closed subspace with
$$T_{a/\overline{a}} H^2 =(\ker T_{\overline{a}/a})^{\perp}=K_I^{\perp}=IH^2$$
(see also Example \ref{ppappap}) which proves the claim.
\end{proof}

\bibliographystyle{plain}

\bibliography{references}

\end{document}